\documentclass{amsart}
\usepackage{graphicx}
\input xy
\xyoption{all}
\newenvironment{theorem}[2][Theorem]{\begin{trivlist}
\item[\hskip \labelsep {\bfseries #1}\hskip \labelsep {\bfseries #2}]}{\end{trivlist}}

\newenvironment{proposition}[2][Proposition]{\begin{trivlist}
\item[\hskip \labelsep {\bfseries #1}\hskip \labelsep {\bfseries #2}]}{\end{trivlist}}

\newtheorem{thm}{Theorem}

\newtheorem{prop}[thm]{Proposition}
\theoremstyle{definition}
\newtheorem{defn}[thm]{Definition}
\theoremstyle{remark}
\newtheorem{rem}[thm]{Remark}
\theoremstyle{definition}

\numberwithin{equation}{section}

\begin{document}

\title[Dualizing complexes and vanishing Koszul homology]{Dualizing complexes and homomorphisms vanishing in Koszul homology}%
\author{Javier Majadas}%
\address{Departamento de \'Algebra, Facultad de Matem\'aticas, Universidad de Santiago de Compostela, E15782 Santiago de Compostela, Spain}%
\email{j.majadas@usc.es}%


\begin{abstract}
Let $C$ be a semidualizing complex over a noetherian local ring $A$. If there exists a local homomorphism with source $A$ satisfying some homological properties, then $C$ is dualizing.
\end{abstract}
\maketitle

\section{Introduction}

There is a number of characterizations of properties (of homological type) of noetherian local rings of positive characteristic in terms (of homological properties) of the Frobenius homomorphism. We start with \cite {Ku}:
\begin{theorem}{(Kunz)}
\textit{Let $A$ be a noetherian local ring containing a field of characteristic $p>0$, and let $\phi :A \rightarrow A, \phi(a)=a^p$ be the Frobenius homomorphism. We denote by $^{\phi}A$ the ring $A$ considered as $A$-module via $\phi$. The following conditions are equivalent:\\*
(i) $A$ is regular\\*
(ii) $^{\phi}A$ is a flat $A$-module.}
\end{theorem}

Some years later, Kunz theorem was improved by Rodicio \cite{Ro} as follows: if the flat dimension $fd_A(^{\phi}A)<\infty$, then $A$ is regular.

So we can think if similar characterizations for the properties complete intersection, Gorenstein and Cohen-Macaulay exist. For complete intersections the result was obtained in \cite {BM}, characterizing complete intersections rings by the finiteness of the complete intersection dimension \cite{AGP} of its Frobenius homomorphism, and a similar characterization was also found for the Cohen-Macaulay property in \cite {TY}. \\

We will examine now in more detail the case of the Gorenstein property. A first result was obtained by Herzog \cite {He} (see also \cite [Theorem 1.1]{Go} and \cite [Proposition 6.1]{TY}):

\begin{theorem}{(Herzog)}
\textit{Let $A$ be a noetherian local ring containing a field of characteristic $p>0$, and let $\phi$ be its Frobenius homomorphism. Assume that $\phi$ is finite. The following conditions are equivalent:\\*
(i) $A$ is Gorenstein\\*
(ii) $Ext_A^i(^{\phi^r}\!A,A)=0$ for all $i>0$ and infinitely many $r>0$.}
\end{theorem}

This result was improved in \cite {IS}, removing in particular the annoying finiteness hypothesis on $\phi$:

\begin{theorem}{(Iyengar, Sather-Wagstaff)}
\textit{Let $A$ be a noetherian local ring containing a field of characteristic $p>0$, and let $\phi$ be its Frobenius homomorphism. The following conditions are equivalent:\\*
(i) $A$ is Gorenstein\\*
(ii) G-dim$_A(^{\phi^r}\!A)<\infty$ for some integer $r>0$.}
\end{theorem}

Here G-dim denotes the Gorenstein dimension introduced by Auslander and Bridger in \cite{AB} (properly speaking, a modification of the original definition using Cohen factorizations \cite [p.254]{AF2}, \cite [Definition 3.3]{IS}).\\

Over the last years, some research was conducted in order to extend these results from the particular case of the Frobenius homomorphism to larger classes of homomorphisms. A first step was to consider contracting endomorphisms. An endomorphism $f$ of a noetherian local ring $(A,\mathfrak{m},k)$ is contracting if for any integer $s>0$ there exist an integer $r>0$ such that $f^r(\mathfrak{m})\subset \mathfrak{m}^s$. The Frobenius homomorphism is an example of contracting endomorphism. If $f$ is a contracting endomorphism on a noetherian local ring $A$, then $A$ must contain a field (of fixed elements), but unlike the case of the Frobenius homomorphism, it can be of characteristic zero. The above results for regularity were extended to contracting endomorphisms in \cite [Proposition 2.6]{KL}. For the complete intersection property they were extended (even in an improved form) in \cite{AIM}, \cite{AHIY}.

The Gorenstein case was studied first in \cite{IS}. In fact, they obtain the theorem stated above as a consequence of the more general:

\begin{theorem}{(Iyengar, Sather-Wagstaff)}
\textit{Let $A$ be a noetherian local ring and $\phi:A \rightarrow A$ a contracting endomorphism. Then the following conditions are equivalent:\\*
(i) $A$ is Gorenstein\\*
(ii) G-dim$_A(^{\phi^r}\!A)<\infty$ for some integer $r>0$.}
\end{theorem}

Subsequently, in \cite{NS} this result was extended to the more general context of G-dimension over a semidualizing complex $C$ as defined in \cite{Ch}. It is obtained in particular:

\begin{theorem}{(Nasseh, Sather-Wagstaff)}
\textit{Let $A$ be a noetherian local ring, $C$ a semidualizing complex over $A$ and $\phi:A \rightarrow A$ a contracting endomorphism. The following conditions are equivalent:\\*
(i) $C$ is a dualizing complex.\\*
(ii) G$_C$-dim$(^{\phi^r}\!A)<\infty$ for infinitely many $r>0$.}
\end{theorem}

This result generalizes the ``classical'' case: the Gorenstein dimension of \cite{AB} is the particular case of G$_C$-dim obtained by taking $C=A$ (which is always a semidualizing complex), and a ring $A$ is Gorenstein if and only if the semidualizing complex $A$ is dualizing. \\

A second step in the extension of these results to larger classes of homomorphisms was initiated in \cite{Ma1}. The purpose in that paper was not so much to extend the known results to larger classes of homomorphisms as to understand what a homomorphism must verify in order to be a ``test homomorphism'' for these properties of local rings. In that paper a new class of homomorphisms, the ones with the $h_2$-vanishing property, was introduced. A local homomorphism $f:(A,\mathfrak{m},k)\rightarrow(B,\mathfrak{n},l)$ of noetherian local rings is said to have the $h_2$-vanishing property if the canonical homomorphism between the first Koszul homology modules associated to minimal sets of generators of their maximal ideals
$$H_1(\mathfrak{m})\otimes_kl \rightarrow H_1(\mathfrak{n})$$
is zero. Any contracting endomorphism has a power which has the $h_2$-vanishing property, but $h_2$-vanishing homomorphisms are not necessarily endomorphisms, and they can be defined on rings that do not contain a field. Moreover, unlike the class of contracting endomorphisms, the class of $h_2$-vanishing homomorphisms contains at once the two main test homomorphisms: the Frobenius endomorphism and the canonical epimorphism of a local ring into its residue field.

In order to see, even in the case of an endomorphism, the difference between $h_2$-vanishing and contracting, consider a complete local ring $(A,\mathfrak{m},k)$ and a contracting endomorphism $\phi$ of $A$. We assume for simplicity that $\phi(\mathfrak{m}) = \phi^1(\mathfrak{m})\subset \mathfrak{m}^2$. Take a regular local ring $(R,\mathfrak{n},k)$ of minimal dimension such that $A=R/I$ (i.e., dim $R$ = emb.dim $A$), and a contracting endomorphism $\varphi$ of $R$ making commutative the diagram \cite [3.2.1, 3.2.4]{MR}

$$
\xymatrix{ R \ar[r]^{\varphi} \ar @{->>}[d] & R \ar @{->>}[d] \\
 A \ar[r]^{\phi} & A }
$$
(details can be seen in \cite [Example 3.ii]{Ma1}).

Then the homomorphism induced by $\phi$
$$H_1(\mathfrak{m})\otimes_k{^{\phi}}k \rightarrow H_1(\mathfrak{m})$$
can be identified with the canonical homomorphism induced by $\varphi$
$$I/\mathfrak{n}I\otimes_k{^{\phi}}k \rightarrow I/\mathfrak{n}I.$$
Since $\varphi$ is contracting, by the Artin-Rees lemma some power of it verifies $\varphi^r (I) \subset \mathfrak{n}I$, and so $\phi^r$ has the $h_2$-vanishing property. But the contracting property is not only a condition on the images of $I$, but on the images of $\mathfrak{n}$. For instance, any local homomorphism which factorizes through a regular local ring has the $h_2$-vanishing property.\\

Our purpose in this paper is to extend the above result of Nasseh and Sather-Wagstaff to $h_2$-vanishing homomorphisms. In order to achieve it, instead of working with G$_C$-dim, we consider a different definition, G$^*_C$-dim (see Definition \ref{*}). Both definitions are related in the same way that Gorenstein dimension G-dim is related to upper Gorenstein dimension G*-dim \cite{Ve}, \cite [\S8]{Av2}. They share the usual properties (see Propositions \ref{k} and 3*), but we do not know if the finiteness of G$_C$-dim is equivalent to the finiteness of G$^*_C$-dim.\\

We obtain:

\begin{theorem}{\ref{th}}
\textit{Let $\varphi : A \rightarrow B$ be a local homomorphism and $C$ a semidualizing complex over $A$. Assume that $\varphi$ has the $h_2$-vanishing property. The following conditions are equivalent:\\*
(i) $C$ is a dualizing complex.\\*
(ii) G$^*_C$-dim$(B)<\infty$.}\\
\end{theorem}

A note on terminology. Since we are interested only in the finiteness of G$_C$-dim and not in its precise value, we use the terminology of derived $C$-reflexivity instead of finite G$_C$-dim.

\section{Notation for complexes}

\textit{All rings in this paper will be noetherian and local.}
\\

We will follow the conventions for complexes generally used in this context (see e.g. \cite {Ch}). For convenience of the reader we will briefly recall some notation. Let $A$ be a ring. A complex of $A$-modules will be a sequence of $A$-module homomorphisms
$$X = ... \rightarrow X_{n+1} \xrightarrow{d_{n+1}} X_{n} \xrightarrow{d_{n}} X_{n-1} \rightarrow ...$$
such that $d_nd_{n+1}=0$ for all $n$. If $m$ is an integer, $\Sigma^mX$ will be the complex with $(\Sigma^mX)_n = X_{n-m}$, $d_n^{\Sigma^mX} = (-1)^md_{n-m}^X$ for all $n$.

The derived category of the category of $A$-modules will be denoted by $\textbf{D}(A)$. For $X,Y \in \textbf{D}(A)$, we will write $X \simeq Y$ if $X$ and $Y$ are isomorphic in $\textbf{D}(A)$, and $X \sim Y$ if $X \simeq \Sigma^mY$ for some integer $m$. Sometimes we will consider an $A$-module as a complex concentrated in degree $0$. The full subcategory of $\textbf{D}(A)$ consisting of complexes homologically finite, that is, complexes $X$ such that $H(X)$ is an $A$-module of finite type, will be denoted by $\textbf{D}_b^f(A)$.

The left derived functor of the tensor product of complexes of $A$-modules will be denoted by $- \otimes^{\textbf{L}}_A -$, and similarly $\textbf{R}{\rm Hom}_A(-,-)$ will denote the right derived functor of the Hom functor on complexes of $A$-modules. We say that a complex $X\in \textbf{D}(A)$ is of finite projective (respectively, injective) dimension if there exists a bounded complex $Y$ (that is, $Y_n=0$ for $|n| \gg 0$) of projective (respectively, injective) modules such that $X \simeq Y$. We will denote it by pd$_A(X) < \infty$ (respectively, id$_A(X) < \infty$).

\section{Derived reflexivity}

Let $X,C \in \textbf{D}_b^f(A)$. We say that $X$ is \textit{derived $C$-reflexive} if $\textbf{R}{\rm Hom}_A(X,C) \in \textbf{D}_b^f(A)$ and the canonical biduality morphism

$$X \rightarrow \textbf{R}{\rm Hom}_A(\textbf{R}{\rm Hom}_A(X,C),C)$$
is an isomorphism in $\textbf{D}_b^f(A)$ \cite [2.7]{Ch}, \cite [\S2]{AIL}.

We will say that $C \in \textbf{D}_b^f(A)$ is a \textit{semidualizing complex} \cite [Definition 2.1]{Ch} if $A$ is derived $C$-reflexive, that is, if the homothety morphism

$$A \rightarrow \textbf{R}{\rm Hom}_A(C,C)$$
is an isomorphism in $\textbf{D}_b^f(A)$. If $C$ is a semidualizing complex and $X \in \textbf{D}_b^f(A)$, then $X$ is derived $C$-reflexive if and only if $X \simeq \textbf{R}{\rm Hom}_A(\textbf{R}{\rm Hom}_A(X,C),C)$ \cite [Theorem 3.3]{AIL}. We will give precise references of all the results we need on derived reflexivity, but the reader may consult \cite{Ch}, \cite{FS} for a systematic study.

A \textit{dualizing} complex is a semidualizing complex of finite injective dimension.

We now introduce a modification of derived $C$-reflexivity. We call it derived $C$-reflexivity*, since it is related to derived $C$-reflexivity in the same way that upper Gorenstein dimension G*-dim is related to Gorenstein dimension G-dim \cite {Ve}, \cite [\S8]{Av2}.

A local homomorphism $(A,\mathfrak{m},k)\rightarrow(R,\mathfrak{p},l)$ is \textit{weakly regular} if it is flat and the closed fiber $R\otimes_Ak$ is a regular local ring. Let $f:(A,\mathfrak{m},k)\rightarrow(B,\mathfrak{n},l)$ be a local homomorphism. A \textit{regular factorization} of $f$ is a factorization $A\xrightarrow{i} R \xrightarrow{p} B$ of $f$ where $i$ is weakly regular and $p$ is surjective. If $B$ is complete, then $f$ has a regular factorization with $R$ complete \cite {AFH}.

\begin{defn}\label{*}
If $C$ is a semidualizing complex over a ring $A$, a \textit{$C$-deformation} of $A$ will be a pair $(Q,E)$ consisting in a surjective homomorphism of (local) rings $Q \rightarrow A$ and a semidualizing complex $E \in \textbf{D}_b^f(Q)$ such that $\textbf{R}{\rm Hom}_Q(A,E) \sim C$. In this case, by \cite [Theorem 6.1 and Observation 2.4]{Ch}, the $Q$-module $A$ is derived $E$-reflexive.

Let $C$ be a semidualizing complex over $A$, and $X \in \textbf{D}_b^f(A)$. We will say that $X$ is \textit{derived $C$-reflexive*} if there exists a weakly regular homomorphism $A \rightarrow A'$ and a $C \otimes_AA'$-deformation $(Q,E)$ of $A'$ (note that $C \otimes_AA' = C \otimes^{\textbf{L}}_AA'$ is a semidualizing complex over $A'$ \cite [Theorem 5.6]{Ch}) such that pd$_Q(X\otimes_AA')<\infty$.

Let $\varphi : A \rightarrow B$ be a local homomorphism, $C$ be a semidualizing complex over $A$, $X \in \textbf{D}_b^f(B)$. We will say that $X$ is \textit{derived $C$-$\varphi$-reflexive*} if there exists a regular factorization $A \rightarrow R \rightarrow \hat{B}$ such that the complex of $R$-modules $X\otimes_B\hat{B}$ is derived $C \otimes_AR$-reflexive*, where $\hat{B}$ is the completion of $B$.

\end{defn}

\begin{prop}\label{rr}
Let $C$ be a semidualizing complex over $A$, and $X \in \textbf{D}_b^f(A)$. If $X$ is derived $C$-reflexive*, then it is derived $C$-reflexive.
\end{prop}

\begin{proof}
Let $A \rightarrow A'$ be a weakly regular homomorphism, $(Q,E)$ a $C':=C \otimes_AA'$-deformation of $A'$ such that pd$_Q(X\otimes_AA')<\infty$. By \cite [Proposition 2.9]{Ch}, $X\otimes_AA'$ is derived $E$-reflexive and then, by \cite [Theorem 6.5]{Ch}, $X\otimes_AA'$ is derived $C'$-reflexive. Then faithfully flat base change \cite [Theorem 5.10]{Ch} gives that $X$ is derived $C$-reflexive.
\end{proof}

We do not know if the reciprocal of Proposition \ref{rr} is true, even in the (classical) case $C=A$. However the usual characterization of dualizing complexes in terms of derived reflexivity of the residue field also remain valid for derived reflexivity* (in the case $C=A$ this is the theorem by Auslander and Bridger saying that a ring $A$ is Gorenstein if and only if the Gorenstein dimension of any module of finite type is finite if and only if the Gorenstein dimension of its residue field is finite \cite [Theorem 4.20 and its proof]{AB}; see also \cite [Theorem 6.1]{IS}):

\begin{prop}\label{k} \cite [Proposition 8.4, Remark 8.5]{Ch} Let $C$ be a semidualizing complex over $A$. The following are equivalent:\\*
(i) $C$ is dualizing.\\*
(ii) Any $X \in \textbf{D}_b^f(A)$ is derived $C$-reflexive.\\*
(iii) The residue field $k$ of $A$ is derived $C$-reflexive.
\end{prop}

\begin{proposition}{\ref{k}*}  Let $C$ be a semidualizing complex over $A$. The following are equivalent:\\*
(i) $C$ is dualizing.\\*
(ii*) Any $X \in \textbf{D}_b^f(A)$ is derived $C$-reflexive*.\\*
(iii*) The residue field $k$ of $A$ is derived $C$-reflexive*.
\end{proposition}

\begin{proof} By Propositions \ref{rr} and \ref{k}, we only have to show (i) $\Rightarrow$ (ii*). Let $\hat{A}$ be the completion of $A$ and $Q \rightarrow \hat{A}$ a surjection where $Q$ is a regular local ring. Let $D$ be a dualizing complex over $Q$ ($D \sim Q$). Then $\textbf{R}{\rm Hom}_Q(\hat{A},D)$ is a dualizing complex over $\hat{A}$ (\cite [V.2.4]{RD} or \cite [Corollary 6.2]{Ch}). Also, $C \otimes_A\hat{A}$ is a dualizing complex over $\hat{A}$ (\cite [V.3.5]{RD}), so $\textbf{R}{\rm Hom}_Q(\hat{A},D) \sim C \otimes_A\hat{A}$ by \cite [V.3.1]{RD}.

Therefore $(Q,D)$ is a $C \otimes_A\hat{A}$-deformation of $\hat{A}$. Since $Q$ is regular, for any $X \in \textbf{D}_b^f(A)$ we have pd$_Q(X\otimes_A \hat{A})<\infty$, and so $X$ is derived $C$-reflexive*.

\end{proof}

This result still holds for derived $C$-$\varphi$-reflexivity*:

\begin{prop}\label{4}
Let $C$ be a semidualizing complex over $A$. The following are equivalent:\\*
(i) $C$ is dualizing.\\*
(ii) For any local homomorphism $\varphi: A \rightarrow B$, any $X \in \textbf{D}_b^f(B)$ is derived $C$-$\varphi$-reflexive*.\\*
(iii) There exists a local homomorphism $\varphi: A \rightarrow B$, such that the residue field $l$ of $B$ is derived $C$-$\varphi$-reflexive*.
\end{prop}
\begin{proof}
(i) $\Rightarrow$ (ii) Let $\varphi: A \rightarrow B$ be a local homomorphism and let $A \rightarrow R \rightarrow \hat{B}$ be a regular factorization with $R$ complete. Since $C$ is a dualizing complex over $A$ and $i$ is flat with Gorenstein (in fact regular) closed fiber, then $C \otimes_A R$ is a dualizing complex over $R$  \cite [Theorem 5.1, Proposition 4.2]{AF1}. Therefore the result follows from Proposition 3*.\\
(iii) $\Rightarrow$ (i) Let $A\xrightarrow{i} R \xrightarrow{p} \hat{B}$ be a regular factorization, $R \rightarrow R'$ a weakly regular homomorphism, and $(Q,E)$ a $C \otimes_A R'$-deformation of $R'$ such that pd$_Q(l\otimes_RR')<\infty$. Since $R \rightarrow R'$ is weakly regular, its closed fiber $l\otimes_RR'$ is regular. Then $Q$ is a regular local ring (it follows e.g. from the change of rings spectral sequence
$$E^2_pq=Tor^{l\otimes_RR'}_p(Tor^Q_q(l\otimes_RR',l'),l') \Rightarrow Tor^Q_q(l',l')$$
where $l'$ is the residue field of $Q$ and $l\otimes_RR'$).

We deduce that id$_Q(E)<\infty$, and so the semidualizing complex $E$ is dualizing. Then $C \otimes_AR' \sim \textbf{R}{\rm Hom}_Q(R',E)$ is also dualizing \cite [V.2.4]{RD}. Since $A \rightarrow R'$ is flat, it is easy to see that $C$ is dualizing (or use the stronger result \cite [Theorem 5.1]{AF1}).

\end{proof}

\begin{defn}\label{h2} \cite [Definition 1]{Ma1}
Let $f:(A,\mathfrak{m},k)\rightarrow(B,\mathfrak{n},l)$ be a local homomorphism. Let $H_*(\mathfrak{m})$ (respectively, $H_*(\mathfrak{n})$) be the Koszul homology associated to a minimal system of generators of the ideal $\mathfrak{m}$ of $A$ (respectively, the ideal $\mathfrak{n}$ of $B$). We say that $f$ has the \textit{$h_2$-vanishing property} if the canonical homomorphism induced by $f$ \\
$$H_1(\mathfrak{m})\otimes_kl \rightarrow H_1(\mathfrak{n})$$
vanishes.
\end{defn}

By \cite [15.12]{An} (see  \cite [2.5.1]{MR}), this homomorphism  between Koszul homology modules can be written in terms of Andr\'e-Quillen homology \cite{An} as the canonical homomorphism
$$H_2(A,k,l) \rightarrow H_2(B,l,l).$$

As we saw in the Introduction, a suitable power of any contracting endomorphism has the $h_2$-vanishing property (in fact, if $f:(A,\mathfrak{m},k)\rightarrow (A,\mathfrak{m},k)$ is a contracting endomorphism, for any integer $n$ there exists an integer $s$ such that $f^s$ has the $h_n$-vanishing property, in the sense that the morphism of functors $H_n(A,k,-) \rightarrow H_n(A,k,-)$ vanishes \cite [Proposition 10]{Ma2}).

\begin{thm}\label{th}
Let $\varphi : A \rightarrow B$ be a local homomorphism and $C$ a semidualizing complex over $A$. Assume that $\varphi$ has the $h_2$-vanishing property. If (and only if) $B$ is derived $C$-$\varphi$-reflexive*, then $C$ is dualizing.
\end{thm}

\begin{proof}
The ``only if'' part is a consequence of Proposition \ref{4}.

Assume then that $B$ is derived $C$-$\varphi$-reflexive*. Consider a diagram of local homomorphisms

$$
\xymatrix{ & & Q \ar[dr] & & \\
 & R  \ar[rr]^{\rho} \ar[dr]^{\pi} & & R' \ar[dr]^{\pi'} & \\
A \ar[ur]^{\alpha} \ar[r]^{\varphi} & B \ar[r]^{\beta} & \hat{B} \ar[rr] & & \hat{B}\otimes_RR' }
$$
where $\alpha$ and $\rho$ are weakly regular, $\pi$ is surjective and $(Q,E)$ is a $C \otimes_AR'$-deformation of $R'$ such that pd$_Q(\hat{B}\otimes_RR')<\infty$. We will see first that $Q$ is a regular local ring repeating an argument in the proof of \cite [Proposition 6]{Ma1}.

Let $l$ be the residue field of $Q$ and $\hat{B}\otimes_RR'$. The commutative square

$$
\xymatrix{A \ar[r]^{\alpha} \ar[d]^{\varphi} & R \ar[d]^{\pi}  \\
B \ar[r]^{\beta} & \hat{B}}
$$
induces a commutative square

$$
\xymatrix{H_2(A,l,l) \ar[r]^{\tilde{\alpha}} \ar[d]^{\tilde{\varphi}} & H_2(R,l,l) \ar[d]^{\tilde{\pi}}  \\
H_2(B,l,l) \ar[r]^{\tilde{\beta}} & H_2(\hat{B},l,l).}
$$

We have $\tilde{\varphi}=0$ since $\varphi$ has the $h_2$-vanishing property (we have used that if $k\rightarrow l$ is a field extension we have $H_n(k,l,l)=0$ for all $n \geq 2$ \cite [7.4]{An}; so if $A\rightarrow k \rightarrow l$ are ring homomorphisms with $k$ and $l$ fields, from the Jacobi-Zariski exact sequence \cite [5.1]{An} we obtain $H_n(A,k,l)=H_n(A,l,l)$ for all $n \geq 2$; finally,  $H_n(A,k,k) \otimes_kl=H_n(A,k,l)$ for all $n$ by \cite [3.20]{An}).

Since $\alpha$ is weakly regular, by \cite[Lemma 5]{Ma1}, $\tilde{\alpha}$ is an isomorphism, and so $\tilde{\pi} =0$. Consider now the commutative square

$$
\xymatrix{  H_2(R,l,l) \ar[d]^0 \ar[r]^{\tilde{\rho}} & H_2(R',l,l) \ar[d]^{\tilde{\pi'}} \\
H_2(\hat{B},l,l) \ar[r]  & H_2(\hat{B}\otimes_RR',l,l). }
$$
Again by \cite[Lemma 5]{Ma1}, $\tilde{\rho}$ is an isomorphism, and then $\tilde{\pi'}=0$. So the composition
$$H_2(Q,l,l) \rightarrow H_2(R',l,l) \xrightarrow{\tilde{\pi'}} H_2(\hat{B}\otimes_RR',l,l)$$
vanishes. But by \cite{Av1}, pd$_Q(\hat{B}\otimes_RR')<\infty$ implies that
$$H_2(Q,l,l) \rightarrow H_2(\hat{B}\otimes_RR',l,l)$$
is injective. Therefore $H_2(Q,l,l)=0$, and then $Q$ is regular by \cite [6.26]{An}.

Now the proof finishes as the proof of Proposition \ref{4}: since id$_Q(E)<\infty$, the semidualizing complex $E$ is dualizing; then $C \otimes_AR' \sim \textbf{R}{\rm Hom}_Q(R',E)$ is also dualizing \cite [V.2.4]{RD} and since $\rho \alpha$ is flat, we deduce that $C$ is dualizing.

\end{proof}

\begin{rem}
If a homomorphism in a composition has $h_2$-vanishing property, then so has the composition. Therefore Theorem \ref{th} can also be stated as follows:

Let $\varphi : A \rightarrow B$ be a local homomorphism and $C$ a semidualizing complex over $A$. Assume that $\varphi$ has the $h_2$-vanishing property. If there exists a local homomorphism $\phi :B \rightarrow S$ such that $S$ is derived $C$-$\phi\varphi$-reflexive*, then $C$ is dualizing.
\end{rem}


\end{document}